\documentclass[a4paper,11pt]{article}

\usepackage{amsmath}
\usepackage{amssymb}
\usepackage{amsthm}
\usepackage{amsfonts}
\usepackage{mathdots}
\usepackage{enumerate}
\usepackage{comment}
\usepackage{enumitem}
\usepackage{hyperref}

\title{The  Ramsey number of loose cycles versus cliques}
\author{Ar\`es M\'eroueh\footnote{Department of Pure Mathematics and Mathematical Statistics, Centre for Mathematical Sciences, Wilberforce Road, Cambridge, CB3 0WB, UK. E-mail: \href{mailto:a.j.meroueh@dpmms.cam.ac.uk}{a.j.meroueh@dpmms.cam.ac.uk}.}  }

\newtheorem{theorem}{Theorem}[section]
\newtheorem{conjecture}[theorem]{Conjecture}
\newtheorem{lemma}[theorem]{Lemma}
\newtheorem{proposition}[theorem]{Proposition}

\newtheorem{definition}[theorem]{Definition}

\newtheorem{claim}[theorem]{Claim}

\newcommand{\N}{\mathbb{N}}

\newcommand{\el}{E_{\text{light}}}
\newcommand{\eh}{E_{\text{heavy}}}
\newcommand{\ph}{P_{\text{heavy}}}
\newcommand{\pl}{P_{\text{light}}}
\newcommand{\red}{H^*}

\begin{document}
\maketitle

\begin{abstract}

Recently Kostochka, Mubayi and Verstra\"ete \cite{mubaykosver} initiated the study of the Ramsey numbers of uniform loose cycles versus cliques. In particular they proved that $R(C^r_3,K^r_n) = \tilde{\theta}(n^{3/2})$ for all fixed $r\geq 3$. For the case of loose cycles of length five they proved that $R(C_5^r,K_n^r)=\Omega((n/\log n)^{5/4})$ and conjectured that $R(C^r_5,K_n^r) = O(n^{5/4})$ for all fixed $r\geq 3$. Our main result is that $R(C_5^3,K_n^3) = O(n^{4/3})$ and more generally for any fixed $l\geq 3$ that $R(C_l^3,K_n^3) = O(n^{1 + 1/\lfloor(l+1)/2 \rfloor})$. 

We also explain why for every fixed $l\geq 5$, $r\geq 4$, $R(C^r_l,K^r_n) = O(n^{1+1/\lfloor l/2 \rfloor})$ if $l$ is odd, which improves upon the result of Collier-Cartaino, Graber and Jiang \cite{colgrajiang} who proved that for every fixed $r\geq 3$, $l\geq 4$, we have $R(C_l^r,K_n^r) = O(n^{1 + 1/(\lfloor l/2 \rfloor-1)})$.  
\end{abstract}
\section{Introduction}
\label{introduction}
A loose cycle of length $l$ is a hypergraph made of $l$ edges $e_1$, $e_2$, \ldots, $e_l$ such that, for any $i,j$, if $j = i+1 \pmod n$ or $j = i-1 \pmod n$ then $|e_i \cap e_j| = 1$ and otherwise $e_i \cap e_j = \emptyset$. For brevity, we shall denote by $C_l$ such a hypergraph. An $r$-uniform loose cycle of length $l$ is a loose cycle of length $l$ whose edges all have size $r$. We shall denote by $C_l^r$ such a hypergraph. This is one of the possible generalizations of graph cycles, and indeed corresponds to a cycle in the graph sense when $r=2$. An $r$-uniform clique of order $n$ is an $r$-uniform hypergraph on $n$ vertices where all the sets of $r$ vertices form an edge. We shall denote by $K_n^r$ such a hypergraph. 

The Ramsey number of an $r$-uniform loose cycle of length $l$ versus an $r$-uniform clique of order $n$, denoted by $R(C_l^r,K_n^r)$, is the least integer $m$ such that whenever the edges of $K_m^r$ are coloured red and blue then $K_m^r$ either contains a red $C_l^r$ (that is, a copy of $C_l^r$ all of whose edges are coloured red) or a blue $K_n^r$ (that is, a copy of $K_n^r$ all of whose edges are coloured blue).  Determining the order of magnitude of $R(C_l^2,K_n^2)$ is a classical problem in graph theory. The best lower bound on $R(C_l^2,K_n^2)$ is due to Bohman and Keevash \cite{bohkee}. They proved that $R(C_l^2,K_n^2) = \Omega(n^{1+1/(l-2)}/\log(n))$. For $l$ even, the best upper bound is due to Caro, Li, Rousseau and Zhang \cite{carolirousseauzhang}; they proved that $R(C_l^2,K_n^2)= O((n/\log(n))^{1 + 1/(l/2 - 1)})$. For $l$ odd, the best upper bound is due independently to Li and Zang \cite{lizang} and to Sudakov \cite{sudakov}. They proved that $R(C_l^2,K_n^2) = O(n^{1 + 1/\lfloor l/2\rfloor}/\log(n)^{1/\lfloor l/2 \rfloor})$. In the light of this, in subsequent discussion we shall refer to $n^{1 + 1/(\lfloor (l-1)/2\rfloor)}$ as ``the graph bound".  

Recently Kostochka, Mubayi and Verstra\"ete initiated the study of $R(C_l^r, K_n^r)$ for $r \geq 3$. In \cite{mubaykosver} they proved the following theorems. 

\begin{theorem}[Kostochka, Mubayi and Verstra\"ete \cite{mubaykosver}]
\label{mubaykosver1}
There exist constants $a,b_r >0$ such that 
$$a\frac{n^{3/2}}{(\log n)^{3/4}} \leq R(C_3^3,K_n^3) \leq b_3n^{3/2}, $$
and for $r\geq 4$,
$$\frac{n^{3/2}}{(\log n)^{3/4+o(1)}} \leq R(C_3^r,K_n^r) \leq b_rn^{3/2}. $$
\end{theorem}
For loose cycles of length five, they proved the following. 
\begin{theorem}[Kostochka, Mubayi and Verstra\"ete \cite{mubaykosver}]
\label{mubaykosver2}
There exist constants $c_r>0$ such that 
$$R(C_5^r,K^r_n) \geq c_r \left(\frac{n}{\log n} \right)^{5/4}. $$
\end{theorem}
They also provided a more general lower bound of the form $R(C_l^r,K_n^r) = \Omega(n^{1+ 1/(3l -1)}) $ for any fixed $r$ and $l$. They made the following conjecture.
\begin{conjecture}[Kostochka, Mubayi and Verstra\"ete \cite{mubaykosver}]
\label{c5conjecture}
For any fixed $r\geq 3$ we have
$$R(C_5^r,K^r_n) = O(n^{5/4}). $$
\end{conjecture}
Collier-Cartaino, Graber and Jiang  \cite{colgrajiang} proved the following. 

\begin{theorem}[Collier-Cartaino, Graber and Jiang \cite{colgrajiang}]
\label{colgrajiang1}
For any fixed $r\geq 3$ and $l\geq 4$ there exist constants $b_{r,l}$ such that 
$$R(C_l^r,K_n^r)\leq b_{r,l}n^{1 + 1/(\lfloor l/2 \rfloor-1)}. $$
\end{theorem}
For $l$ even, they are able to improve this bound by a polylogarithmic factor. We notice that this proves that the graph bound holds for $R(C_l^r,K_n^r)$ when $l$ is even but falls short when $l$ is odd. We also notice that for $r=3$ and $l=5$ this says that $R(C_5^3,K_n^3) = O(n^2)$.  

Our main result is the following.
\begin{theorem}
\label{maintheorem}
There exists a constant $c_{3,5}$ such that $R(C_5^3,K_n^3) \leq c_{3,5}n^{4/3}$ and more generally for any fixed $l\geq 3$ there exists a constant $c_{3,l}$ such that $R(C_l^3,K_n^3)\leq c_{3,l} n^{1 + 1/\lfloor (l+1)/2 \rfloor}$. 
\end{theorem}
 For $l=3$ notice that this follows from Theorem \ref{mubaykosver1}. We believe that this result is interesting for two reasons. First, it brings the bound of $O(n^2)$ on $R(C_5^3,K_n^3)$ due to Collier-Cartaino, Graber and Jiang down to $O(n^{4/3})$, a bound much closer to Conjecture \ref{c5conjecture}. Secondly, the bound we obtain beats the best known upper bound on $R(C_l^2,K_n^2)$ i.e., the graph bound, for each $l\geq 3$ by an order of magnitude. In fact, we would expect that one should be able to prove that for every fixed $r \geq 3$, $l\geq 3$, there exist $\epsilon_{r,l}>0$ and $c_{r,l}>0$ such that $R(C_l^r,K_n^r) \leq c_{r,l} n^{1 + 1/\lfloor (l-1)/2 \rfloor - \epsilon_{r,l}}$. Thus Theorem \ref{mubaykosver1} settles this question for $l=3$ and $r\geq 3$, while Theorem \ref{maintheorem} settles it for $r=3$ and $l\geq 4$. However for $r\geq 4$ and $l\geq 4$ the methods we use do not seem to generalize in a straightforward way (See Section \ref{remarks} for more details.) The proof of Theorem \ref{maintheorem} relies on generalizing various ideas found in \cite{colgrajiang}, \cite{mubaykosver}, \cite{lizang} and \cite{sudakov} together with some new ones of our own. Let us now state our second result. 

\begin{theorem}
\label{graphbound}
For any fixed $r\geq 3$, $l\geq 5$, there exists $c_{r,l}>0$ such that $R(C_l^r,K_n^r) \leq c_{r,l} n^{1 + 1/\lfloor l/2 \rfloor}$ when $l$ is odd.
\end{theorem}

The main point of Theorem \ref{graphbound} is that it essentially proves that the graph bound also holds for $R(C_l^r,K_n^r)$ when $l$ is odd, thus completing the result of Theorem \ref{colgrajiang1}. We made no attempt at improving this by a polylogarithmic factor as we do not believe the exponent to be correct. We shall only sketch the proof of Theorem \ref{graphbound} as most of the ideas necessary to prove it will already have been developed to prove Theorem \ref{maintheorem}. This sketch also serves to highlight how we can beat the graph bound if $r=3$. 
  
\section{Notation and Tools}
In this section we review some basic notation and definitions related to hypergraphs. We also state a few results which we will need in order to prove Theorem \ref{maintheorem}.

For $a,b \in \N$, $[b]$ denotes the set  $\{1,2,\ldots,b\}$ and $[a,b]$ denotes the set $\{a, a+1, \ldots, b \}$.

A hypergraph $H=(A,B)$ is is a pair of finite sets $A$, $B$  such that $B$ is a set of subsets of $A$. The elements of $A$ will be referred to as the \emph{vertices} of $H$ and those of $B$ as the \emph{edges} of $H$. For a given hypergraph $H=(A,B)$ we let $V(H)$ denote the set of vertices (that is, $A$) and let $E(H)$ denote the set of edges (that is, $B$). Often when it is clear that say $u$, $v$ are vertices of $V$, we shall write $uv$ to mean the edge $\{u,v \}$. A hypergraph $H$ is said to be $r$-uniform if all the elements of $B$ have the same size $r$. We shall also call an $r$-uniform hypergraph an $r$-graph, for short. Notice that when $r=2$ we get the classical definition of a loopless graph. If $a,b$ are two integers with $a<b$ then an $[a,b]$-graph means a hypergraph whose edges each have size lying in $[a,b]$. 

Given $v \in V(H)$ the degree of $v$ in $H$, denoted by $d(v)$, is the number of edges of $H$ containing $v$. The average degree of $H$, denoted by $d$, is the quantity $(\sum_{v\in V(H)}d(v))/n$. 

Given a hypergraph $H$ and a subset $X$ of $V(H)$, a subhypergraph of $H$ is a hypergraph with vertex $X \subseteq V(H)$ and set of edges a subset of $E(H)$ made of edges contained in $X$. Given $X\subseteq V(H)$, the \emph{induced} subhypergraph of $H$ on vertex set $X$, denoted by $H[X]$, is the hypergraph $(X, \{e\in E(H):e\subseteq X \})$. A hypergraph $H$ not containing (an isomorphic copy of) a  hypergraph $F$ as a subhypergraph is said to be $F$-free.  

Given two hypergraph $H_1$ and $H_2$ we shall let $H_1 + H_2$ denote the hypergraph with vertex set $V(H_1) \cup V(H_2)$ and set of edges $E(H_1) \cup E(H_2)$. 

A hypergraph $H$ is said to be \emph{simple} if for all $e,f\in E(H)$, if $e\neq f$ then $|e\cap f| \leq 1$. Notice that loose cycles as defined in Section \ref{introduction} are simple hypergraphs, whereas $r$-uniform cliques for $r\geq 3$ and $n\geq r+1$ are not. 

A path of length $l$ is a hypergraph $P$ such that $E(P) = \{e_1,e_2,\ldots,e_l \}$ and for all $i\neq j$, if $i<j$ then $e_i\cap e_j = \emptyset$ unless $i\leq l-1$ and $j= i+1$, in which case we require $e_i \cap e_j \neq \emptyset$. We also require all the edges of the path to have size at least 2. If $P$ is a simple path of length $l$ and $a\in e_1\backslash e_2$ and $b \in e_l \backslash e_{l-1}$ then we say that $P$ joins $a$ to $b$. 

A hypergraph $H$ is said to be \emph{$k$-vertex-colourable} (or \emph{$k$-colourable} for brevity) if there exists a map $c:V(H)\longrightarrow [k]$ such that for any $e\in E(H)$ there exist $a,b \in e$ such that $c(a) \neq c(b)$. This condition on $c$ makes it a \emph{proper} colouring: we remark this because later we shall also refer to any map $f:V(H)\longrightarrow [k]$ as a colouring. Suppose $V(H)$ is totally ordered by some ordering $<$. A path $P = \{e_1,e_2,\ldots, e_l\}$ of $H$ is said to be increasing if $a \leq b$ whenever $a\in e_i$ and $b\in e_j$ for some $i$, $j$ with $i<j$. Pluh\'ar \cite{pluhar} proved the following.

\begin{proposition}[Pluh\'ar \cite{pluhar}]
\label{pathcolouring}
A hypergraph $H$ is $k$-colourable if and only if there exists a total ordering $<$ of $V(H)$ for which there is no increasing simple path of length $k$ in $H$. 
\end{proposition} 
Thus, in particular, if $H$ contains no simple path of length $k$ then $H$ is $k$-colourable. 

A set $X \subseteq V(H)$ is said to be \emph{independent} in $H$ if no edge of $H$ is contained in $X$. The independence number of $H$, denoted by $\alpha(H)$, is the maximal size of an independent subset of $V(H)$. An easy observation is that if $H$ is $k$-colourable then $\alpha(H) \geq |V(H)|/k$. The following well-known result of Spencer \cite{spencer} gives another way of bounding the independence number of a hypergraph. It is an analogue of Tur\'an's Theorem \cite{turan} for $r\geq 3$. 

\begin{proposition}[Spencer \cite{spencer}]
\label{lowdegreebound}
Let $H$ be an $r$-uniform hypergraph with average degree $d$. Then $\alpha(H) \geq 0.5n/d^{\frac{1}{r-1}}$. 
\end{proposition} 

The reason why we are interested in bounding the independence number of a hypergraph is that bounding $R(C_l^3,K_n^3)$ from above is clearly equivalent to bounding the independence number of $C_l^3$-free 3-graphs from below. 

\section{Light and Heavy Pairs}
Throughout this section and the next ones, $l$ denotes a fixed integer which is at least 3.

In this section we introduce the concept of light and heavy pairs of vertices in a 3-graph $H$.  These first appeared in the proof of the upper bound of Theorem \ref{mubaykosver1} in \cite{mubaykosver}. We generalize the ideas found there as they will be equally useful when looking at $C_l^3$-free 3-graphs.  
\begin{definition}
\label{lightheavy} Let $H$ be a 3-graph. A pair $ab$ of vertices of $H$ is called \emph{light} in $H$ if it is contained in less than $2l-2$ edges of $H$. It is called \emph{heavy} in $H$ otherwise. We let $\pl(H)$ be the set of all the pairs which are light in $H$ and $\ph(H)$ the set of all the pairs which are heavy in $H$. 
\end{definition}
\begin{definition}
Let $H$ be a 3-graph. An edge $e$ of $H$ is called \emph{heavy} if it contains a pair which is heavy in $H$. Otherwise, it is called \emph{light}. We let $\eh(H)$ be the set of heavy edges of $H$ and $\el(H)$ the set of light ones. 
\end{definition}

\begin{definition}
\label{reduced}
Let $H$ be a 3-graph. The \emph{reduced} hypergraph of $H$, denoted by $\red$, is the hypergraph with vertex set $V(H)$ and set of edges $\el(H)\cup \ph(H)$.
\end{definition}
Observe that $\red$ is a $[2,3]$-graph. 
\begin{lemma}
\label{nocl}
Let $H$ be a 3-graph. If $H + \red$  contains a $C_l$ then $H$ contains a $C^3_l$. 
\end{lemma}
\begin{proof}
Let $C$ be a $C_l$ in $H + \red$. Let $A$ be the set of edges of $C$ which belong to $E(H)$ and let $B$ be the set of edges of $H$ which belong to $\ph(H)$. If $B = \emptyset$ then $C$ is already a $C_l^3$ in $H$ and there is nothing to prove. If $B\neq \emptyset$ then let $uv\in B$. We have $|V(C)\backslash \{u,v\}| = l + |A| -2\leq 2l -3 $. Therefore since $uv $ is contained in $2l-2$ edges of $H$ at least one of them, call it $e$, doesn't meet $V(C) \backslash \{u,v \}$. Then if we remove $uv$ from $C$ and add $e$ we obtain a $C_l$ in $H+\red$ with one less edge lying in $\ph(H)$. Repeat this operation until $B = \emptyset$. 
\end{proof}

\begin{lemma}
\label{lightpairsdecomposition}
Let $H$ be a $C_l^3$-free 3-graph. Then there exist $l-2$ hypergraphs $H_1$, $H_2$,\ldots, $H_{l-2}$ with vertex set $V(H)$ such that $E(H)=\bigcup_{i=1}^{l-2}E(H_i)$ and for each $i\in [l-2]$, each edge of $H_i$ contains a light pair in $H_i$.  
\end{lemma}

\begin{proof}
Suppose to the contrary that there exists a $C_l^3$-free 3-graph $H$ for which the conclusion of the lemma does not hold. Let $J_0 = H$ and let $G_0$ be the 2-graph with vertex set $V(H)$ and set of edges $\ph(H)$. For $i\in [l-2]$, let $H_i$ be the hypergraph with vertex set $V(H)$ and edges those elements of $E(J_{i-1})$ containing an element of $\pl(J_{i-1})$. Let $J_i$ be the hypergraph with vertex set $V(H)$ and edges those elements of $E(J_{i-1})$ such that the three pairs of vertices that they contain all belong to $E(G_{i-1})$. Finally, let $G_i$ be the 2-graph with vertex set $V(H)$ and set of edges $\ph(H_i)$.

We shall now prove by induction on $k$ the claim that for $k\in [0,l-3]$, $G_{l-3-k}$ contains a $C^2_{k+3}$. First consider the base case $k=0$. By our assumption on $H$, $E(H)\neq \bigcup_{i=1}^{l-2}E(H_i)$ and therefore $E(J_{l-2})\neq \emptyset$ since $E(H) =  \bigcup_{i=1}^{l-2}E(H_i)\cup E(J_{l-2})$. Let $x_1x_2x_3 \in E(J_{l-2})$. By definition of $J_{l-2}$, the pairs $x_1x_2$, $x_2x_3$ and $x_3x_1$ are heavy in $J_{l-3}$ so that $G_{l-3}$ contains the $C^2_3$ formed by these three pairs. Suppose now that the induction hypothesis holds for some $k\in [0,l-4]$ and let us prove it holds for $k+1$. Let $C = x_1x_2\ldots x_{k+3}$ be a $C^2_{k+3}$ in $G_{l-3-k}$. Since the pair $x_{k+3}x_1$ is heavy in $J_{l-3-k}$, it is contained in at least $2l-2$ edges of $J_{l-3-k}$. One of these therefore does not meet any vertex of $C$ other than $x_{k+3}$ and $x_1$ since $|C| \leq l-1$; let $x_{k+3}x_1y$ be such an edge. The pairs $x_{k+3}y$ and $yx_1$ are then heavy in $J_{l-3-(k+1)}$ by definition of $J_{l-3-k}$. It is also clear that the pairs $x_1x_2$, $x_2x_3$, \ldots, $x_{k+2}x_{k+3}$, being heavy in $J_{l-3-k}$, are also heavy in $J_{l-3-(k+1)}$. Thus $x_1x_2\cdots x_{k+3}y$ is a $C^2_{k+4}$ in $G_{l-3-(k+1)}$, as required. This finishes the proof of the claim.

So there exists a $C^2_l$ in $\ph(H)$. By Lemma \ref{nocl} there exists a $C^3_l$ in $H$, which is a contradiction. 
\end{proof}

Let us remark that the quantity $l-2$ in Lemma \ref{lightpairsdecomposition} is by no means tight, but since we do not care about the constant implicit in Theorem \ref{maintheorem}, this shall be enough for our purposes. Indeed, in the rest of this paper we shall not seek to optimize the constants appearing in the various lemmas. The next lemma will play a very important role in our argument.  

\begin{lemma}
\label{lightedgesdecomposition}
Let $H$ be a 3-graph. Then $\el(H)$ is the union of at most $6l-11$ simple hypergraphs. 
\end{lemma}
\begin{proof}
Let $G$ be the 2-graph with vertex set $\el(H)$ and where $e$ is joined to $f$ by an edge if $|e \cap f|=2$. Then this graph has maximum degree at most $6l-12$ (for if $e\in \el(H)$ has degree at least $6l-11$ in this hypergraph then one of the 3 pairs of vertices contained in $e$ is contained in at least $2l-3$ edges of $H$ other than $e$ and hence is a heavy pair in $H$, a contradiction). Therefore there exists a proper vertex colouring of $G$ on $6l-11$ colours; now clearly each colour class of this colouring forms a simple hypergraph with vertex set $V(H)$.  
\end{proof}

The idea of Lemma \ref{lightedgesdecomposition} essentially appears in the proof of Lemma 7.7 of \cite{colgrajiang}. Without delving into the details, let us point out that the difference here is that the authors of \cite{colgrajiang} were only seeking a large simple subhypergraph of some carefully chosen $C_l$-free hypergraph, whereas in our case the fact that each edge of $\el(H)$ is contained in one of the simple hypergraphs given by Lemma \ref{lightedgesdecomposition} is vital.

\section{Extenders}

\label{extendersection}

In a $C_l^2$-free $2$-graph, the neighbourhood of a vertex $v$ contains no path of length $l-2$, hence is $(l-2)$-colourable by Proposition \ref{pathcolouring}, and so contains an independent set of size at least $|\Gamma(v)|/(l-2)$. Furthermore, there is always a vertex whose neighbourhood is at least as large as the average degree $d$ of the graph. Thus an elementary argument to find a large independent set in a $C_l^2$-free 2-graph is, provided $d$ is large, to find it as a subset of a neighbourhood of a vertex of large degree and, if $d$ is small, to apply Tur\'an's Theorem \cite{turan} which guarantees the existence of an independent set of size at least $n/(1+d)$ where $n$ is the number of vertices of the 2-graph in question. For 3-graphs, the situation is a bit more complicated. Extenders, which are introduced in Definition \ref{extender} below, will play the same role as the neighbourhood of a vertex in the argument we just gave. The various lemmas of this section are aimed at proving that extenders satisfy all the properties that are required to make the argument work, and we shall use some ideas from \cite{mubaykosver}. 

\begin{definition}
\label{extender}
A pair $(X,Y)$ of disjoint subsets of $V(H)$ is called an \emph{extender} if 
\begin{enumerate}
\item
For any $u,v \in X$ with $u\neq v$ and any set $S\subseteq V(H)\backslash (\{u,v \} \cup Y)$ of size at most $2l-5$ there exists a simple path of length two in $H$ joining $u$ to $v$ and which contains no element of $S$;
\item
$|Y| \leq 2|X|$.
\end{enumerate}
The size of the extender $(X,Y)$ is defined to be $|X|$. 
\end{definition}

Thus an extender is a generalization of the neighbourhood of a vertex in a 2-graph in the sense that, if $v$ is a vertex of a 2-graph $G$, then by letting $X = \Gamma(v)$ and $Y = \{ v\}$ we see that $(X,Y)$ certainly satisfies the requirement of being an extender. The difference for a 3-graph is that $Y$ will typically have size much larger than one and also that proving that large extenders exist is not as straighforward as in the 2-graph case. The following is an easy corollary to Lemma \ref{nocl}.

\begin{lemma}
\label{nolpath}
Let $H$ be a $C_l^3$-free 3-graph and let $(X,Y)$ be an extender in $H$. Then for any $u,v\in X$ with $u\neq v$, there does not exist a simple path of length $l-2$ in $(H + \red )[V(H)\backslash Y]$ joining $u$ to $v$. 
\end{lemma}

\begin{proof}
If there were such a path $P$, then if we let $S = V(P)\backslash \{ u,v\}$, clearly $|S| \leq 2l-5$ and so there exists a simple path of length two joining $u$ to $v$ and containing no other vertex of $P$; thus it forms a $C_l$ in $H + \red$. But then by Lemma \ref{nocl} $H$ contains a $C_l^3$, a contradiction.  
\end{proof}

The next lemma proves that in a $C_l^3$-free 3-graph $H$ there exists an extender of size at least a constant times the average degree of $H$. 

\begin{lemma}
\label{existence}
Let $H$ be a $C_l^3$-free $3$-graph of average degree $d$. Then there exists an extender $(X,Y)$ such that $|X| \geq d/(24l^2)$.   
\end{lemma}
\begin{proof}
Let $H_1$, $H_2$, \ldots, $H_{l-2}$ be hypergraphs such that $E(H) = \bigcup_{i=1}^{l-2}E(H_i)$ and for each $i$, each edge of $H_i$ contains a pair that is light in $H_i$. Such a collection of hypergraphs exists by Lemma \ref{lightedgesdecomposition}. By the pigeonhole principle there exists $i\in[l-2]$ such that $|E(H_i)|\geq |E(H)|/(l-2)$. We may assume without loss of generality that $i=1$. For $j=1,2,3$ let $H_{1j}$ be the hypergraph with vertex set $V(H)$ and consisting of those edges of $H_1$ containing precisely $j$ light pairs in $H_1$. We consider two different cases. 
\begin{description}[leftmargin=*]
\item[Case 1: $|E(H_{11})|\geq |E(H_1)|/2$.] 

Each edge of $H_{11}$ gives two pairs $(x,e)$ of a vertex $v$ and an edge $e$ of $H_1$ such that $v$ is contained in the (unique) light pair of $H_1$ contained in $e$. Therefore there exists a vertex $v$ contained in the light pair of at least $2|E(H_{11})|/n$  edges of $H_{11}$. But $2|E(H_{11})|/n \geq |E(H_1)|/n \geq E(H)/(n(l-2)) = d/(3(l-2))$. List these edges as $vx_iy_i$: $i=1,\ldots, m$, $m \geq d/(3(l-2))$, so that the pairs $vx_i$ are light and the pairs $vy_i$ are heavy in $H_1$. In particular, no $y_i$ can occur as $x_j$ for any $j$. As the pairs $vx_i$ are light, at least $m/(2l-3)$ of the $x_i$'s are pairwise distinct; so we let $X$ be a set of $m/(2l-3)$ pairwise distinct $x_i$'s and let $Y$ be $\{v\} \cup \{y_i:x_i \in X\} $. We claim that the pair $(X,Y)$ is an extender. It is clear that $|Y| \leq 2|X|$, so let us check that the first condition holds. Let $x_i,x_j\in X$. Let $S\subseteq V(H)\backslash (\{x_i,x_j \} \cup Y)$, $|S| \leq 2l-5$. If $y_i \neq y_j$ then $\{vx_iy_i, vx_jy_j\}$ forms a required path of length two. If $y_i = y_j$ then since $x_jy_i$ is heavy in $H$ there exists $z\in V(H)\backslash (S \cup \{v,x_i\})$ such that $y_ix_jz \in E(H)$. Then $\{vx_iy_i,y_ix_jz\}$ forms the required path of length two. 

\item[Case 2: $|E(H_{12})\cup E(H_{13})|\geq |E(H_1)|/2$.] Each edge of $E(H_{12})\cup E(H_{13})$ defines at least one pair $(v,e)$ of a vertex $v$ and an edge $e$ of $H_1$ such that $v$ belongs to two light pairs of $e$. Thus there exists a vertex $v$ contained in two light pairs of at least  $|E(H_1)|/(2n) \geq d/(6(l-2))$ edges of $H_1$. List these edges as $vx_iy_i$, $1\leq i\leq m;  m\geq d/(6(l-2))$, so that the pairs $vx_i$ and $vy_i$ are light for all $i$. The fact that these pairs are light implies that we may find at least $m/(4l-7)$ pairs $x_iy_i$ which are pairwise disjoint; without loss of generality the pairs $x_iy_i$ are pairwise disjoint for $i \in [1, \lceil m/(4l-7)\rceil]$. We let $X = \{ x_i:\, i\in [\lceil m/(4l-7)\rceil]\}$  and we let $Y = \{v \} \cup \{ y_i:\, i\in [\lceil m/(4l-7)\rceil]\}$. It is clear that $|Y| \leq 2|X|$ and for any $i\neq j$ and any $S\subseteq V(H)\backslash (\{x_i,x_j \} \cup Y)$, $ \{ vx_iy_i, vx_jy_j\} $ is a path of length two not meeting $S$ which joins $x_i$ to $x_j$.  \qedhere
\end{description}
\end{proof}

We are now ready to prove Lemma \ref{niceindependence}, which is the main result of this section. It says that if $(X,Y)$ is an extender, then $X$ contains a large independent set in $\red$. Observe that being independent in $\red$ is stronger than being independent in $H$, i.e. that any set independent in $\red$ is also independent in $H$. This is because any edge of $H$ contains an edge of $\red$. The reason why we wish to find an independent set in $\red$ rather than merely $H$ is that in the proof of Theorem \ref{maintheorem} we will seek an independent set in $H$ which is made of subsets of several disjoint extenders $(X_1,Y_1)$, $(X_2,Y_2)$,\ldots (for a suitable definition of ``disjoint"), and so we shall need not only that each subset of each extender is independent in $H$ but also that there are no edges \emph{between} these subsets. This is where the extra information given by the lemma will be useful. 

\begin{lemma}
\label{niceindependence}
Let $H$ be a $C_l^3$-free 3-graph and let $(X,Y)$ be an extender in $H$. Then $X$ contains a subset $Z$ which
\begin{enumerate}
\item
is independent in $\red$;
\item
has size at least $|X|/(l-2)$.
\end{enumerate}
\end{lemma}

\begin{proof}[Proof of Lemma \ref{niceindependence}]
By Lemma \ref{nolpath} $\red[X]$ contains no simple path of length $l-2$ and hence by Proposition \ref{pathcolouring} is $(l-2)$-colourable. Hence $X$ contains a set $Z$ which is independent in $\red$ and has size at least $|X|/(l-2)$.
\end{proof}

\section{Neighbourhoods of Extenders}

In a $C_l^2$-free 2-graph, the $i^{\text{th}}$  neighbourhood of a vertex $v$ (that is, the vertices at distance precisely $i$ from $v$) is also $(l-2)$-colourable provided $i\leq \lfloor (l-1)/2 \rfloor$ (see Erd\"{o}s, Faudree, Rousseau and Schelp \cite{erfaurouschelp}). The arguments of \cite{lizang} and \cite{sudakov} use this in order to bound $R(C^2_l,K_n^2)$. Likewise, the $i^{\text{th}}$ neighbourhood of an extender, if carefully defined, will be useful in the proof of Theorem \ref{maintheorem}. 
 
\begin{definition}
Let $H$ be a 3-graph and let $(X,Y)$ be an extender in $H$. Let $S\subseteq V(H)$, $X\cup Y \subseteq S $.  For $i\in \N$ and $v\in S\backslash (X \cup Y)$, the distance between $v$ and $(X,Y)$ within $S$, denoted by $d_S(v,(X,Y))$, is the minimal length of a simple path $P$ in $\red[S\backslash Y]$ joining $v$ to a vertex $x$ of $X$. 

The $i^{\text{th}}$ neighbourhoud of $(X,Y)$ within $S$, denoted by $\Gamma_{S,i}(X,Y)$, is the set $\{v\in V(H)\backslash Y:\, d_S(v,(X,Y)) = i \}$. We also let $\Gamma_{S,0}(X,Y) = X$.  Finally, we define $\Gamma_{S,\leq i}(X,Y)$ to be $\bigcup_{j\in [0,i]}\Gamma_{S,j}(X,Y)$.

\end{definition}
So the $i^{\text{th}}$ neighbourhood of $(X,Y)$ within $S$, for $i\geq 1$, consists of those vertices of $S$ which do not lie in $X$ and which can be joined to a vertex of $X$ by a path of $\red[S]$ of length $i$ which does not meet $Y$, but by no such path of length less than $i$. Notice that a key element of this definition is that we are looking at paths in $\red$, not $H$. Our motivation for introducing neighbourhoods of extenders is the following lemma, which  plays the same role for neighbourhoods of extenders as Lemma \ref{niceindependence} does for extenders. Again the stronger statement that $\Gamma_{S,i}(X,Y)$ contains a large independent set in $\red$ rather than $H$ will be necessary in the proof of Theorem \ref{maintheorem}. 

\begin{lemma}
\label{neighbourhoodindependence}
Let $H$ be a 3-uniform, $C_l^3$-free hypergraph. Let $(X,Y)$ be an extender in $H$. Let $S\subseteq V(H) $, $X\cup Y \subseteq S $. Let $i\in [m-1]$ where $m = \lfloor (l-1)/2 \rfloor$ and $m\geq 2$. Then $\Gamma_{S,i}(X,Y)$ contains a set $Z$ which
\begin{enumerate}
\item
is independent in $\red$; 

\item
is such that $|Z|\geq |\Gamma_{S,i}(X,Y)|/b_l$ where $b_l = (6l-10)^{m - 1} \cdot (2m - 1)^{2m - 1}\cdot (l-2)^{3^{2m - 1}}$.
\end{enumerate}
\end{lemma}

Let us point out that the parameter $S$ in the definition of the $i^{\text{th}}$ neighbourhood of an extender plays no important role in this lemma. It will only be required later on in the proof of Theorem \ref{maintheorem}. What the proof of Lemma \ref{neighbourhoodindependence} actually shows is that amongst $n$ vertices each joined to $X$ by a path in $\red$ of length $i$ not meeting $Y$, we may find $n/b_l$ vertices which form an independent set in $\red$, where $b_l$ is a constant whose value does not matter to us. 

\begin{proof}[Proof of Lemma \ref{neighbourhoodindependence}]

For the sake of clarity, the proof will contain several subclaims. 

By Lemma \ref{lightedgesdecomposition}, $\el(H)$ can be partitioned into $6l-11$ simple hypergraphs which we denote by $H_1$, $H_2$, \ldots, $H_{6l-11}$. Furthermore let $H_0 = \ph(H)$. Thus $\red$ is a $[2,3]$-graph whose set of edges is partitioned by the simple hypergraphs $H_0$, $H_1$, $H_2$,\ldots, $H_{6l-11}$. 

For each $v\in \Gamma_{S,i}(X,Y)$ let $P_v$ be a simple path of length $i$ in $\red[S\backslash Y]$ which joins $v$ to some vertex $x_v$ of $X$. For the rest of the proof, the choice of $P_v$ and $x_v$ is fixed for each $v\in \Gamma_{S,i}(X,Y)$. We also fix an enumeration of the edges of $P$ as  $E(P_v) = \{ f_1^v, f_2^v,\ldots, f_i^v \}$ with $x_v\in f_1^v$ and $v\in f_i^v$ and $f_k^v \cap f_{k+1}^v \neq \emptyset$ for all $k\in [i-1]$, as well as an enumeration of the vertices of $P$ as $V(P) = \{x_1^v,x_2^v,\ldots,x_{|V(P)|}^v \}$ so that $x_1^v = x_v$ and $x^v_{|V(P)|} = v$, and if  $x_r^v \in f^v_j$, $x_s^v \in f^v_k$ with $j < k$ then $r \leq s$. It is easy to see that both of the enumerations we just described exist and are unique, because the path $P_v$ is simple and its edges have size no more than 3.   

The \emph{type} of $P_v$ is the tuple $(t_k)_{k=1}^i$ which is such that for each $k\in [i]$, $f_k^v \in H_{t_k}$. Clearly for any $v$, $P_v$ has one of $(6l-10)^i$ possible types. Therefore we have our first claim.

\begin{claim}
\label{z1}
There exists a subset $Z_1$ of $\Gamma_{S,i}(X,Y)$ of size at least $|\Gamma_{S,i}(X,Y)|/(6l-10)^i$ such that all the paths $P_v$ for $v\in Z_1$ are of the same type.
\end{claim}

Since all the paths $P_v$ for $v\in Z_1$ have the same type, it is clear that they contain the same number of vertices, and we denote by $p$ this quantity, so $p\leq 2i+1$.  Let $c:V(H)\longrightarrow [2i+1]$ be a (not necessarily proper) colouring of the vertices of $V(H)$. We say that a path $P_v$ for $v\in Z_1$ is \emph{rainbow} with respect to $c$ if $c(x_k^v) = k$ for all $k \in [p]$. Our second claim is the following.

\begin{claim}
\label{z2}
There exists a colouring $c:V(H)\longrightarrow [2i+1]$ and a subset $Z_2$ of $Z_1$ of size at least $|Z_1|/(2i+1)^{2i+1}$ such that $P_v$ is rainbow with respect to $c$ for all $v\in Z_2$. 
\end{claim}
\begin{proof}[Proof of Claim \ref{z2}]
Let $c$ be the colouring obtained by attributing to each vertex of $H$ one of the $2i+1$ possible colours uniformly at random and independently of other vertices. Then for any $v\in Z_1$ the probability that $P_v$ is rainbow with respect to $c$ is precisely $1/(2i+1)^{p}\geq 1/(2i+1)^{2i+1}$. Thus the expected number of vertices $v$ such that $P_v$ is rainbow with respect to $c$ is at least $|Z_1|/(2i+1)^{2i+1}$ and so there exists a choice of $c$ and of $Z_2$ such that the claim holds. 
\end{proof}

In order to prove that $\Gamma_{S,i}(X,Y)$ contains a large independent set in $\red$, we wish to apply Proposition \ref{pathcolouring}, in the same way as we did in the proof of lemma \ref{niceindependence} above. Ideally we would like to say that ``$\Gamma_{S,i}(X,Y)$ cannot contain a simple path $P$ of length $l-2$ since otherwise some subpath of $P$, call it $P'$, would join two vertices $a$, $b$ of $\Gamma_{S,i}(X,Y)$ such that there is a subpath $P'_a$ of $P_a$ and a subpath $P'_b$ of $P_b$, such that  $P' + P'_a + P'_b$  forms a $C_l$, a contradiction". There are several difficulties which prevent us from doing this directly, however (What happens if $P'_a$ and $P'_b$ meet each other more than once? What happens if $P'_a$ or $P'_b$ meets $P'$ more than once? How do we ensure that $P'_a + P'_b + P'$ has length $l$?) The fact that the paths $P_v$ for $v\in Z_2$ are rainbow goes a long way towards resolving these problems. Indeed, notice that if $u,v \in Z_2$ then $V(P_u) \cap V(P_v)= \{x^u_i: x^u_i = x^v_i, i\in [p] \}$ since $P_v$ and $P_u$ are rainbow. This would guarantee in the discussion above, for example, that $P_a$ and $P_b$ intersect $P$ only once each. But this is not enough, as the main problem we are faced with is to guarantee that $P'_a + P'_b + P'$ has length $l$. This is why we introduce the \emph{class} of an edge of $\red[Z_2]$ in what follows. 

Order the vertices of $H$ arbitrarily and let $<$ be the chosen total ordering. Given $e\in E(\red[Z_2])$, let $a$ be the smallest vertex under $<$ which is contained in $e$ and let $b$ be the largest one. The \emph{class} of $e$ is the tuple $(m_k)_{k=1}^p$ whose coordinates take values in $\{0,1,2\}$ such that $m_k = 0$ if $x_k^a < x_k^b$, $m_k=1$ if $x_k^a = x_k^b$ and $m_k = 2$ if $x_k^a > x_k^b$. So there are $3^p\leq 3^{2i+1}$ possible classes for an element of $E(\red[Z_2])$, and for each possible class $t$ of an element of $E(\red[Z_2])$ we let $J_t$ be the hypergraph with vertex set $Z_2$ and whose edges are all the elements of $E(\red[Z_2])$ of type $t$. Our third claim is the following.

\begin{claim}
\label{classcolouring}
For each element $t$ of $\{0,1,2\}^{p}$, $J_t$ is $(l-2)$-colourable.
\end{claim}

\begin{proof}[Proof of Claim \ref{classcolouring}]
By Proposition \ref{pathcolouring} it is enough to show that $J_t$ does not contain an increasing simple path of length $l-2$ with respect to $<$. Suppose, to the contrary, that it did contain such a path $P$. Let $a$ be the smallest vertex of $P$ and $b$ the largest one (with respect to $<$). Write $E(P) = \{e_1,e_2,\ldots,  e_{l-2} \}$ with $a\in e_1$, $b\in e_{l-2}$ and $e_j \cap e_{j+1}\neq \emptyset$ for each $j \in [l-3]$. For each $j\in  [l-3]$ let $u_j$ be the vertex of $P$ belonging to $e_j \cap e_{j+1}$. Furthermore let $u_0 = a$, $u_{l-2} = b$. 

Recall that by definition of $J_t$, each of $e_1$, $e_2$,\ldots, $e_{l-2}$ has class $t$. Suppose that $P_{u_0}$ and $P_{u_1}$ do not intersect. Then, since $e_1$ has type $t$, we have $m_k = 0$ or $m_k = 2$ for any $k\in [i]$. Let $k\in[i]$. If $m_k = 0$ then we know, since each of $e_1$, $e_2$, \ldots, $e_{l-2}$ has class $t$, that $x_k^{u_0} < x_k^{u_1} < \cdots < x_k^{u_{l-2}} $. Likewise, if $m_k = 2$ we then know that $x_k^{u_0} > x_k^{u_1} > \cdots > x_k^{u_{l-2}} $. So in either case, for any $k \in [p]$, the $x_k^{u_j}$'s, $j\in [0,l-2]$, are pairwise distinct. But, since the paths $P_{u_0}, P_{u_1}, \ldots, P_{u_{l-2}}$ are rainbow with respect to $c$, this implies that $P_{u_0}\cap P_{u_j} = \emptyset$ for all $ j \in [l-2]$. Thus in particular $P_{u_0} \cap P_{u_{l-2i-2}} = \emptyset$. Hence $P_{u_0} + \{ e_1, e_2,\ldots, e_{l-2i-2} \} + P_{u_{l-2i-2}}$ is a simple path of length $l-2$ in $\red$ not meeting $Y$ and joining two vertices of $X$, a contradiction to Lemma \ref{nolpath}. 
 
Thus we may assume that $P_{u_0}$ and $P_{u_1}$ do intersect in some vertex. This means that some entry of $t$ is equal to one; let $q$ be largest such that $t_q = 1$, and let $h$ be largest such that $x^{u_0}_q\in f^{u_0}_h$. Let $A = \{ f^{u_0}_h, f^{u_0}_{h+1}, \ldots, f^{u_0}_i \}$, $B = \{ f^{u_{l-2i+2h-2}}_h, f^{u_{l-2i+2h-2}}_{h+1}, \ldots, f^{u_{l-2i+2h-2}}_i \} $ and $D = \{e_1,e_2,\ldots, e_{l-2i+2h-2} \}$. We shall now prove that $C := A + D + B $ is a $C_l$ in $H + \red$, which is a contradiction by Lemma \ref{nocl} and thus finishes the proof of the claim. Let $W = V(A) \cap V(B)$. To prove that $C$ is a $C_l$ it is enough to prove that $W= \{x^{u_0}_q \}$, since the rainbow property of $P_{u_0}$ and $P_{u_{l-2i+2h-2}}$ implies that $A$ and $B$ only intersect $D$ in $u_0$ and $u_{l-2i+2h-2}$ respectively. Since $t_q = 1$ we have $x^{u_0}_q = x^{u_1}_q$, $x^{u_1}_q = x^{u_2}_q$, \ldots, $ x^{u_{l-2i+2h-3}}_q = x^{u_{l-2i+2h-2}}_q$ and so $x^{u_0}_q = x^{u_{l-2i+2h-2}}_q$. Thus $x^{u_0}_q \in W$. Suppose now that there exists $u \in W$, $u\neq x_q^{u_0}$. So $u = x^{u_0}_w$ for some $w\neq q$. If $w>q$ then as $t_w\neq 1$ we either have $x^{u_0}_w < x^{u_1}_w < \ldots < x^{u_{l-2i+2h-2}}_w$ or $x^{u_0}_w > x^{u_1}_w > \ldots > x^{u_{l-2i+2h-2}}_w$ and so either way $x^{u_0}_w \neq x^{_{l-2i+2h-2}}_w$. But then as the paths $P_{u_0}$ and $P_{u_{l-2i+2h-2}}$ are rainbow, $x^{u_0}_w \not\in W$, a contradiction. Thus $w < q$. Then $w\in f^{u_0}_{h}$ since $w <q$, $w\in A$ and $x_q^{u_0} \in f^{u_0}_{h}$. Also $w\in f^{u_{l-2i+2h-2}}_{h}$ by the rainbow property of $P_{u_0}$ and $P_{u_{l-2i+2h-2}}$. Hence $\{x_q^{u_0},x_w^{u_0} \}\subseteq f^{u_0}_h\cap f^{u_{l-2i+2h-2}}_h$  and so $|f^{u_0}_{h}\cap f^{u_{l-2i+2h-2}}_h|\geq 2$. Notice also that $f^{u_0}_{h}\neq f^{u_{l-2i+2h-2}}_h$ (if not then by definition of $q$ this can only happen if $x_q^{u_0}\in f^{u_0}_{h+1}$ and this contradicts the definition of $h$). But $f^{u_0}_{h}$ and $f^{u_{l-2i+2h-2}}_h$ belong to the same simple hypergraph $H_j$ (for some $j$) since $P_{u_0}$ and $P_{u_{l-2i+2h-2}}$ are of the same type, and this is a contradiction. The claim is proved.
\end{proof}

Since $J_t$ is $(l-2)$-colourable for any class $t$, we see that $\red[Z_2]$ is $(l-2)^{3^{2i+1}}$-colourable, since any edge of $\red[Z_2]$ belongs to $J_t$ for some $t$ and there are at most $3^{2i+1}$ values of $t$.  We therefore have the following.

\begin{claim}
\label{z3}

$Z_2$ contains a set $Z$ of size at least $|Z_2|/(l-2)^{3^{2i+1}}$ which is independent in $\red$.
\end{claim}

Lemma \ref{neighbourhoodindependence} is now proved.
\end{proof}

\section{Proof of Theorem \ref{maintheorem}}
\label{maintheoremsection}
Before proving Theorem \ref{maintheorem} in earnest, let us briefly explain how one might find a large independent set in a $C^2_l$-free 2-graph. Let $m = \lfloor (l-1)/2 \rfloor $. As mentioned at the beginning of Section \ref{extendersection}, a simple bound on the independence number of a $C^2_l$-free 2-graph $G$ can be found by considering the average degree $d$ of $G$. However, when $l\geq 5$, we can do better. Indeed in a ``typical" 2-graph $G$ of average degree $d$, we expect the $m^{\text{th}}$ neighbourhood of a vertex to have size about $d^m$. This is why, if $d$ is large, it might be a better idea to seek an independent set in $\Gamma_m(v)$ rather than $\Gamma(v)$ (since $\Gamma_m(v)$ is $(l-2)$-colourable), and when $d$ is small to still apply Tur\'an's theorem as explained in Section  \ref{extendersection}. This in itself is not a valid argument, obviously, since it is not the case that $|\Gamma_m(v)|\geq d^m $ (for some $v$) in every graph. But, if the $m^{\text{th}}$ neighbourhood of a vertex is bounded in $G$, then it is a better idea to look at the $i^{\text{th}}$ neighbourhood of a vertex for $1\leq i\leq m-1$. In fact, both \cite{lizang} and \cite{sudakov} either find a large independent set which is the union of large subsets of $i^{\text{th}}$ neighbourhoods of vertices (where some care is needed to make sure that there is no edge between these sets), or find an induced subgraph of $G$ of small average degree, where one can then apply Tur\'an's Theorem to find a large independent set. We adopt the same strategy, and are able to improve upon the graph bound because we are able to put ourselves in the position of applying Proposition \ref{lowdegreebound} with $r=3$ rather than Tur\'an's Theorem. The main difficulty for hypergraphs is to introduce useful definitions of what is meant by a ``neighbourhood" and this was the subject of the previous sections. 

Let $H$ be a $C_l^3$-free 3-graph on $n$ vertices. The statement of Theorem \ref{maintheorem} is equivalent to proving that the independence number of $H$ is at least a constant times $n^{(m+1)/(m+2)}$ where $m =\lfloor (l-1)/2 \rfloor$, and this is what we shall prove. Let us notice that if $l = 3$ or $l=4$, i.e. $m = 1$, then the theorem can be proved as follows: either $H$ contains an extender of size at least $n^{2/3}$ hence contains an independent set of size at least $n^{2/3}/(l-2)$ by Lemma \ref{niceindependence} or has average degree no more than $24l^2n^{2/3}$ by Lemma \ref{existence} and hence contains an independent set of size at least $0.5n/(24l^2n^{2/3})^{1/2} = n^{2/3}/(4\sqrt{6}l)$ by Proposition \ref{lowdegreebound} with $r=3$. Thus henceforth we shall assume $l\geq 5$ and so $m\geq 2$.  

Notice that for every extender $(X,Y)$ in $H$ such that $|X| \geq n^{2/(m+2)}$ and every $S\subseteq V(H)$ containing $X\cup Y$, we may assume that there exists an integer $i\in [0,m-2]$ such that $|\Gamma_{S,i+1}(X,Y)| \leq n^{1/(m+2)}|\Gamma_{S,i}(X,Y)|$ for otherwise we have $|\Gamma_{S,m-1}(X,Y)|\geq n^{(m+1)/(m+2)}$ and so by Lemma \ref{neighbourhoodindependence} we can find an independent set of size at least $n^{(m+1)/(m+2)}/b_l$ in $H$.

Consider the following procedure producing a sequence $((X_k,Y_k))_{k\in[t]}$ of extenders of $H$, a sequence $(S_k)_{k\in [t+1]}$ of subsets of $H$, and a sequence $(i_k)_{k\in [t]}$ of elements of $[0,m-2]$.
\begin{itemize}
\item
Initially, we let $S_1 = V(H)$. If there is no extender $(X,Y)$ in $H$ with $|X|\geq n^{2/(m+2)}$ then we STOP. Otherwise we let $(X_1,Y_1)$ be an extender in $H$ with $|X_1| \geq n^{2/(m+2)}$, and we let $i_1$ be the least $i\in [0,m-2]$ such that $|\Gamma_{S_1,i+1}(X_1,Y_1)| \leq n^{1/(m+2)}|\Gamma_{S_1,i}(X_1,Y_1)|$. 
\item
Having obtained sequences $((X_j,Y_j))_{j\in[k]}$, $(S_j)_{j\in [k]}$ and $(i_j)_{j\in [k]}$, we let $$S_{k+1} = V(H) \backslash \left(\bigcup_{j\in [k]} \left(Y_j \cup \Gamma_{S_j,\leq i_j+1}(X_j,Y_j) \right)\right).$$ If there is no extender $(X,Y)$ in $H[S_{k+1}]$ with $|X| \geq n^{2/(m+2)}$ then we STOP. Otherwise, we let $(X_{k+1},Y_{k+1})$ be such an extender. This is clearly an extender in $H$, as $H[S_{k+1}]$ is a subhypergraph of $H$. We let $i_{k+1}$ be the least $i\in [0,m-2]$ such that $|\Gamma_{S_{k+1},i+1}(X_{k+1},Y_{k+1})| \leq n^{1/(m+2)}|\Gamma_{S_{k+1},i}(X_{k+1},Y_{k+1})|$. 
\end{itemize}
Clearly, this procedure must terminate. The sequence produced has the following important property. 
\begin{multline}
\label{eq1}
\text{For every $k_1, k_2 \in [t]$ with $k_1 < k_2$, } \\ \left(Y_{k_1} \cup \Gamma_{S_{k_1},\leq i_{k_1}+1}(X_{k_1},Y_{k_1})\right)\cap \Gamma_{S_{k_2}, i_{k_2}}(X_{k_2},Y_{k_2})=\emptyset. 
\end{multline}

Suppose first that $|S_{t+1}| \geq n/2$. Then $H[S_{t+1}]$ contains no extender of size at least $n^{2/(m+2)}$. By Lemma \ref{existence} this implies that the average degree of $H[S_{t+1}]$ is no more than $24l^2n^{2/(m+2)}$. Hence by Proposition \ref{lowdegreebound} $\alpha(H[S_{t+1}]) \geq 0.5(n/2)/(24l^2n^{2/(m+2)})^{1/2} = n^{(m+1)/(m+2)}/(8\sqrt{6}l)$. But clearly $\alpha(H) \geq \alpha(H[S_{t+1}])$ and so we are done. 

So we may assume that that $|S_{t+1}|\leq n/2$. We shall find a large set which is independent in $\red$ (rather than $H$). As mentioned above in Section \ref{extendersection} such a set is also independent in $H$. The reason why we look for an independent set in $\red$ rather than $H$ is that neighbourhoods of extenders are defined in terms of paths in $\red$ rather than $H$.  Let $$T =  \bigcup_{k\in [t]}\left(Y_k \cup \Gamma_{S_k,\leq i_k+1}(X_k,Y_k) \right).$$
Since $S_{t+1}\leq n/2$ we have $|T| \geq n/2$. For each $k\in [t]$, by the definition of $i_k$ we have $|\Gamma_{S_k,i_k}(X_k,Y_k)| \geq |\Gamma_{S_{k},i}(X_k,Y_k)|$ for any $i\leq i_k$ and $|\Gamma_{S_k,i_k}(X_k,Y_k)|\geq |\Gamma_{S_k,i_k+1}(X_k,Y_k)|/n^{1/(m+2)}$. Thus,
$$ |\Gamma_{S_k,i_k}(X_k,Y_k)|\geq \left|Y_k \cup \Gamma_{S_k,\leq i_k+1}(X_k,Y_k)\right|/((m+2)n^{1/(m+2)})$$
(Recall that by definition of an extender, $|Y_k| \leq 2|X_k|$.) By Lemma \ref{neighbourhoodindependence} (Lemma \ref{niceindependence} when $i_k = 0$), for every $k\in [t]$, $\Gamma_{S_k,i_k}(X_k,Y_k)$ contains a set $Z_k$ of size at least $|\Gamma_{S_k,i_k}(X_k,Y_k)|/b_l$ which is independent in $\red$. Let $Z = \cup_{k\in[t]}Z_k $, so that $|Z|\geq |T| / (b_l(m+2)n^{1/(m+2)}) \geq n^{(m+1)/(m+2)}/(2b_l(m+2))$. We shall find $W\subseteq Z$ of size at least $|Z|/2^{2m-1}$ which is independent in $\red$, and this shall finish the proof of Theorem \ref{maintheorem}.  

Let $v\in Z$. Then $v\in Z_k$ for some $k\in [t]$. Suppose $i_k \geq 1$. Then as $v\in \Gamma_{S_k,i_k}(X_k, Y_k)$, there exists a simple path $P_v$ in $\red[S_k]$ which joins $v$ to an element of $X_k$, which is disjoint from $Y_k$, and which has length $i_k$. As in the proof of Lemma \ref{neighbourhoodindependence}, we select one such path $P_v$ for each $v\in W$ and fix our choice for the rest of the proof. If $i_k = 0$ we let $P_v = \{v \}$. Let $c:V(H) \longrightarrow \{\text{blue}, \text{red} \}$ be a (not necessarily proper) 2-colouring of the vertices of $H$. We say that $P_v$ is \emph{well-coloured} by $c$ if the colour given to $v$ is red and the colour of any other vertex of $P_v$ is blue.  If $c$ is a colouring obtained by randomly and uniformly assigning the colour blue or red to each vertex of $H$, independently of other vertices, then the probability that $P_v$ is well-coloured is clearly $(1/2)^{|V(P_v)|}$, which is at least $(1/2)^{2i_k+1}$.  Therefore by a mere expectation argument there exists a colouring $c$ and $W\subseteq Z$ with $|W|\geq |Z|/2^{2m-1}$ such that for each $v\in W$, $P_v$ is well-coloured. 

Let us check that $W$ is independent in $\red$. Indeed suppose to the contrary that it isn't. Let $e\in \red[W]$. As each $Z_k$ is independent in $\red$, $e$ meets at least two distinct $Z_k$'s. Let $k_1$ be minimal such that $e\cap Z_{k_1}\neq \emptyset$. Let $k_2 > k_1$ be such that $e\cap Z_{k_2}\neq \emptyset$. Let $a\in e\cap Z_{k_1}$ and let $b\in e\cap Z_{k_2}$. By definition of $k_1$ and the fact that $S_{k_1} \supseteq S_{k_1 + 1} \supseteq S_{k_1 + 2} \supseteq \cdots$, we have $e\subseteq S_{k_1}$, in other words $e\in E(\red[S_{k_1}])$. By \eqref{eq1} we have that $e\cap Y_{k_1} = \emptyset$ (Indeed, \eqref{eq1} implies that that all the elements of $e$ not lying in $\Gamma_{S_{k_1}, i_{k_1}}(X_{k_1},Y_{k_1})$ do not lie in $Y_{k_1}$.) Thus in fact $e\in E(\red[S_{k_1}\backslash Y_{k_1}])$. 

We now consider two different cases: $i_{k_1} = 0$ and $i_{k_1} \geq 1$. Consider first the case $i_{k_1} = 0$. Since $e$ joins $b$ to $a \in X_{k_1}$ and $e\in E(\red[S_{k_1} \backslash Y_{k_1}])$, we have $b \in \Gamma_{S_{k_1},\leq 1}(X_{k_1},Y_{k_1})$.  This is a contradiction to \eqref{eq1} given that $b \in \Gamma_{S_{k_2},k_2}(X_{k_2},Y_{k_2})$. Hence we may assume that $i_{k_1} \geq 1$. In this case, as $P_a$ is well-coloured by $c$, its sole red vertex is $a$, and since $e\subseteq W$, all its vertices are coloured red. Hence $P_a \cap e = \{a \}$, so that the path $P_a + e$ is simple. But then $P_a + e$, being a simple path in $\red[S_{k_1}\backslash Y_{k_1}] $ of length $i_{k_1} + 1$ joining $b$ to an element of $X_{k_1}$, shows that $b \in \Gamma_{S_{k_1},\leq i_{k_1} + 1}(X_{k_1},Y_{k_1})$. This is a contradiction to \eqref{eq1} and finishes the proof of Theorem \ref{maintheorem}.  

\section{Proof of Theorem \ref{graphbound}}
Here we shall give a sketch of the proof of Theorem \ref{graphbound}. Let $r$, $l$ be fixed integers with $r\geq 2$ and $l$ odd, $l\geq 5$. Let $H$ be a $C_l^r$-free $r$-graph on $n$ vertices. 

The first step of the proof is to show that there exists a $C_l$-free $[2,r]$-graph $H'$ such that $V(H') = V(H)$, $\alpha (H) \geq \alpha(H')$ and $E(H')$ can be partitioned into a constant number of simple hypergraphs. We shall use the same reduction ideas as in Lemma 7.5 and Lemma 7.6 of \cite{colgrajiang}, but we go one step further by partitioning $E(H')$ into simple hypergraphs. A sunflower $\mathcal{S}$ with core $C$ is a collection $\mathcal{S}$ of sets such that $e,f\in \mathcal{S}$ and $e\neq f$ implies $e\cap f = C$. If $|\mathcal{S}| = p$ and $|C| = a$ then we say that $\mathcal{S}$ is an $(a,p)$-sunflower. The Sunflower Lemma is the following statement.

\begin{proposition}[P. Erd\"{o}s, R. Rado \cite{erdosrado}]
\label{sunflower}
Let $\mathcal{F}$ be a collection of sets of size at most $r$. If $|\mathcal{F}| \geq r!(p-1)^r$ then $\mathcal{F}$ contains a sunflower with $p$ members. 
\end{proposition}
To find $H'$, one can now proceed as follows. Suppose $H$ contains an $(a,rl)$-sunflower $\mathcal{S}$ for $a \geq 2$. Let $C$ be the core of $\mathcal{S}$. Remove all the edges of $H$ containing $C$ and then add $C$ to $H$. It can be checked that this does not increase the independence number of $H$ or create a $C_l$. Repeat this procedure until no $(a,rl)$-sunflower exists in $H$ with $a\geq 2$. Call the resulting hypergraph $H'$. 

Clearly no pair of vertices of $H'$ is contained in $r!(rl-1)^r$ edges of $H'$ else by Proposition \ref{sunflower} $H'$ would contain an $(a,rl)$-sunflower with $a\geq 2$. This implies that the graph with vertex set $E(H')$ where $e$ and $f$ are adjacent if $|e\cap f|\geq 2$, has maximal degree less than ${r\choose 2} r!(rl-1)^r $, and hence by the same argument as in Lemma \ref{lightedgesdecomposition} we see that $E(H')$ can be partitioned into ${r\choose 2} r!(rl-1)^r$ simple hypergraphs, which we denote by $H'_1$, $H'_2$, \ldots, $H'_p$, where $p = {r\choose 2} r!(rl-1)^r$.

Since $\alpha(H) \geq \alpha(H')$, it is enough to find in $H'$ an independent set of size at least a constant times $n^{m/(m+1)}$ where $m = \lfloor l/2 \rfloor$. We shall proceed  as in the proof of Theorem \ref{maintheorem}. 

Extenders are defined in the same way as before, except that now we require $|Y|\leq (r-1)|X|$. In a simple $[2,r]$-graph there is always an extender $(X,Y)$ of size at least $d/r$ where $d$ is the average degree of the $[2,r]$-graph in question (consider the neighbourhood of a vertex of maximal degree in the hypergraph). Hence in any subhypergraph of $H'$ there is an extender of size at least a constant times the average degree of that subhypergraph, because its edges can be partitioned into a constant number of simple hypergraphs. 

The $i^\text{th}$ neighbourhood of an extender within a set $S$ is defined as before except that we regard $(H')^*$ as being equal to $H'$ (Because ``$H'$ is already reduced".) Thus as before, for an extender $(X,Y)$, if $0 \leq i\leq m-1$ then $\Gamma_{S,i}(X,Y)$ contains a large independent set in $H'$. To see why, it suffices to follow the argument of Lemma \ref{neighbourhoodindependence} where the hypergraphs $H_0$, $H_1$, \ldots, $H_{6l-11}$ are replaced by $H'_1$, $H'_2$, \ldots, $H'_p$ and $\red$ by $H'$. Here we modify the definition of the type of a path $P_v$ slightly: for $P_u$ and $P_v$ to be of the same type, we require as before that corresponding edges along $P_u$ and $P_v$ belong to the same simple hypergraph but we now also require that they have the same size. This only affects the bound obtained on $|Z|$ in Lemma \ref{neighbourhoodindependence} by a constant factor.  

Finally we proceed as in Section \ref{maintheoremsection}, but with different parameters. Namely, an extender $(X,Y)$  is added to the sequence if $|X| \geq n^{1/(m+1)}$, and $i_{k+1}$ is the least $i\in [0,m-2]$ such that $|\Gamma_{S_{k+1},i+1}(X_{k+1},Y_{k+1})|\\ \leq n^{1/(m+1)}|\Gamma_{S_{k+1}, i}(X_{k+1},Y_{k+1})|$ (As before, it is easy to see that we may assume without loss of generality that $i_{k+1}$ exists.)  If $|S_{t+1}|\geq n/2$, then $H'[S_{t+1}]$ is a $[2,r]$-graph of average degree no more than a constant times $n^{1/(m+1)}$. By applying Proposition \ref{lowdegreebound} with $r=2$ to the 2-graph with vertex set $S_{t+1}$ and set of edges $\{\{u,v\} \subseteq S_{t+1}: \exists e\in H'[S_{t+1}] \text{ s.t. } \{u,v\} \subseteq e \}$ we find an independent set of size at least a constant times $n^{m/(m+1)}$ in $H'$. If $|S_{t+1}|\leq n/2$ then we follow the rest of the proof of Theorem \ref{maintheorem} (where $(H')^* = H'$); no $Z_k$ can contain an edge of $H'$ by Lemma \ref{neighbourhoodindependence} (and Lemma \ref{niceindependence}) and the existence of an edge of $H'$ containing vertices from two distinct $Z_k$'s again eventually implies a violation of \eqref{eq1}. 

\section{Further Remarks}
\label{remarks}
We mentioned in the introduction that we believe that one should be able, for any fixed $r\geq 3$ and $l\geq 3$, to beat the graph bound for $R(C_l^r,K_n^r)$ by an order of magnitude. However, the obvious generalization of the methods we use fails for $r\geq 4$ and $l\geq 4$. An example of a 4-uniform hypergraph where our attempts fail  is the following. Let $V(H) = [2n]$ and let $E(H) = \{ \{i,i+n,j,j+n \}: i,j \in [n], i\neq j \}$. It is clear that this hypergraph contains no loose cycle since any two of its edges meet in 0 or 2 vertices. Naturally, this hypergraph does contain a very large independent set, but there are no useful extenders in this hypergraph, so there is no obvious argument which can make use of Proposition \ref{lowdegreebound}.  It might still be possible to improve upon the graph bound for some specific values of $r$ and $l$, but as no straighforward generalization seems possible we did not cover this here. The first natural case to consider is $R(C_5^4,K_n^4)$, where we do not know how to beat the graph bound of $O(n^{3/2})$.

\section{Acknowledgement}

The author wishes to thank Andrew Thomason for his very helpful comments and advice which resulted in a much better presentation of this paper.

\end{document}